\documentclass[12pt, a4paper]{amsart}
\pdfoutput=1
\usepackage{amsmath,amsfonts,amssymb,amsthm,color}  
\usepackage{mathrsfs} 

\usepackage[utf8]{inputenc}
\usepackage{graphicx}

\theoremstyle{plain}
\newtheorem{theorem}{Theorem}[section]

\newtheorem{lemma}[theorem]{Lemma}
\newtheorem{proposition}[theorem]{Proposition}

\theoremstyle{definition}

\newtheorem{remark}[theorem]{Remark}

\numberwithin{equation}{section}
\newtheorem*{theorem*}{Theorem}

\newcommand{\R}{{\mathbb R}}

\def\XXint#1#2#3{{\setbox0=\hbox{$#1{#2#3}{\int}$}
\vcenter{\hbox{$#2#3$}}\kern-.5\wd0}}

\DeclareMathOperator{\supp}{supp}

\providecommand{\norm}[1]{ \lVert#1  \rVert}

\author{Olli Saari and Christoph Thiele}
\title[Lipschitz hyperbolic cross]{Lipschitz linearization of the maximal hyperbolic cross multiplier}

\address{Olli Saari,	
	Department of Mathematics and Systems Analysis, 
	Aalto University School of Science,
	FI-00076 Aalto, 
	Finland
} \email{olli.saari@aalto.fi}

\address{Christoph Thiele,	
	Institute of Mathematics, 
	University of Bonn,
	Endenicher Allee 60, 53115, Bonn,
	Germany
} \email{thiele@math.uni-bonn.de}

\subjclass[2010]{Primary: 42B15, 42B25} 

\keywords{Hyperbolic cross, multipliers, maximal functions, square functions}

\allowdisplaybreaks

\begin{document}

\begin{abstract}
We study the linearized maximal operator associated with dilates of the hyperbolic cross multiplier in dimension two. Assuming a Lipschitz condition and a lower bound on the linearizing function, we obtain $L^{p}(\mathbb{R}^{2}) \to L^{p}(\mathbb{R}^{2})$ bounds for all $1<p <\infty$. We discuss various related results. 
\end{abstract}

\maketitle

\section{Introduction}

Given a bounded measurable function $m$ on $\R^d$, we define the multiplier 
operator $T_m$ by setting $\widehat{T_m f} = m \widehat{f}$ for every Schwartz function $f$.  We also define for $t>0$ the dilated operators 
\[ \widehat{ T_{m,t} f} (\eta) = m(t \eta) \widehat{f}(\eta) , \]
the corresponding maximal operator
\[  T_{m,*} f (x) = \sup_{t>0} |T_{m,t} f(x)| ,\]
and the corresponding linearized maximal operator 
\[  T_{m,V} f (x) = T_{m,V(x)} f(x) ,\]
for a measurable function $V:\R^d \to (0,\infty)$.
For example, in the case of Bochner-Riesz multipliers $m$, such operators 
have been studied in \cite{Carbery1983}.

In this paper we are concerned with the family of hyperbolic cross multipliers, see \cite{Carbery1984,DTU2016,El-Kohen1983}. A hyperbolic cross multiplier on $\R^2$ takes the form
$m(|\eta \xi|)$ for some compactly supported smooth function $m$. 
By the Marcinkiewicz multiplier theorem, this multiplier is bounded on $L^{p}(\mathbb{R}^{2})$ for all $1<p <\infty$. Whether the corresponding maximal operator satisfies any $L^p$ bounds remains open, this question has been posed
in \cite{DT1985,SW2011}. In general, a supremum over an infinite family of Marcinkiewicz or H\"ormander-Mikhlin multipliers does not admit such bounds \cite{CGHS2005,GHS2006,Honzik2014}. 

We propose a Lipschitz assumption on the linearizing function $V$,
in analogy to conjectures for directional operators, see \cite{LaceyLi2010}.
To be efficient in the otherwise dilation invariant setting, 
this assumption has to come with a truncation of the maximal operator, 
which in the theorems below takes either the form of a truncation of the 
multiplier on the Fourier side or a cutoff for the linearization function.

The hyperbolic cross has spikes in vertical and horizontal directions.
Splitting the multiplier into separate pieces, each spike needs a 
Lipschitz assumption only in one variable corresponding to the direction
of the spike. 

\begin{theorem}
\label{thm:main_theorem}
Let $m \in C^{3}(\mathbb{R})$ be a compactly supported function of one variable and 
$\beta \in \mathbb{R}$. 
Let $V : \mathbb{R}^{2} \to [0,\infty)$ be one-Lipschitz in the first variable, that is
\begin{equation}
\label{eq:intro1}
| V(x,y) - V(x',y)| \leq |x-x'|  
\end{equation}
for all $x,x',y \in \mathbb{R}$. Then 
\[T_{m,V} \Pi_{\beta} f(x,y) := \iint_{ \{ |\eta|^{\beta} \leq 1 \}}  m(V(x,y) |\xi | |\eta  |^{\beta})  \widehat{f}(\xi, \eta) e^{2\pi i  (x \xi + y \eta)} d\eta d\xi  \]
satisfies 
\[\norm{T_{m,V}\Pi_\beta f}_{L^{p}(\mathbb{R}^{2})} \leq  C(p,m,\beta) \norm{f}_{L^{p}(\mathbb{R}^{2})}. \]
for all $1 < p < \infty$ and all Schwartz functions $f$.  Here we denoted
$$\Pi_\beta= T_{1_{ \{|\eta| ^\beta\le 1\}}}.$$

\end{theorem}

The case $\beta=0$ is fairly easy, the maximal operator can then be dominated by the Hardy--Littlewood
maximal operator in the first variable. The Lipschitz
assumption is not needed in this case. For $\beta\neq 0$, the truncation to
$|\eta| \le 1$ or $|\eta| \ge 1$ depending on the sign of $\beta$ is the one alluded to above that is necessary
to make the Lipschitz assumption effective. Without truncation one could change the Lipschitz constant by scaling
the otherwise scaling invariant problem, and by a limiting process, letting the constant tend to infinity, get rid of the Lipschitz condition.

The next theorem has a truncation in the
linearizing function $V$, and one needs the Lipschitz condition 
only for distant points. For simplicity we formulate it only
in the case $\beta=1$ of the hyperbolic cross.

\begin{theorem}
\label{thm:main_theorem_2}
Let $m \in C^{3}(\mathbb{R})$ be a compactly supported function of one variable. Let $L> 0$ and let $V : \mathbb{R}^{2} \to [L^{2},\infty)$ satisfy
\begin{equation}
\label{eq:intro2}
| V(z) - V(z')| \leq \max( L^{2} , L |z-z'|  ) 
\end{equation}
for all $z,z' \in \mathbb{R}^{2}$. Then 
\[T_{m,V}f(x,y) = \iint_{\mathbb{R}^{2}} m(V(x,y) |\xi| | \eta| ) \widehat{f}(\xi, \eta) e^{2\pi i  (x \xi + y \eta)} d\xi d\eta  \]
satisfies 
\[\norm{T_{m,V}f}_{L^{p}(\mathbb{R}^{2})} \leq  C(p,m) \norm{f}_{L^{p}(\mathbb{R}^{2})}. \]
for all $1 < p < \infty$ and all Schwartz functions $f$.  
\end{theorem}
 
It actually suffices to demand \eqref{eq:intro1} just for $L=|z-z'|$,  because for other pairs of points one can reduce to this case by a suitable chain of points connecting $z$ and $z'$. Questions of sharp modulus of continuity of $V$
therefore do not arise in this Theorem.


In the case of directional operators, the multiplier $m$ may be chosen to be
a characteristic function of a half line, essentially yielding what is called the 
directional Hilbert transform. In the case of the hyperbolic cross, 
a characteristic function $m$ is problematic since the boundary of the
hyperbolic cross is curved, and with the methods of the
celebrated disc multiplier theorem \cite{Fefferman1971} one sees that the operator $T_m$ cannot be bounded in $L^{p}(\mathbb{R}^{2})$ if $p \neq 2$. When $p=2$, this obstruction disappears, and we have the following variant of our result:
\begin{theorem}
\label{thm:3}
Let $V : \mathbb{R}^{2} \to \{2^{k} \}_{k \in \mathbb{Z}_{+}}$ be such that $V = 2^{ \lfloor \log_2 v \rfloor }$ for a function $v$ that is one-Lipschitz. Then
\[T_V f(x,y) = \iint_{\mathbb{R}^{2}}  1_{[0,1)}(|\eta|) (v(x,y) |\xi \eta |) \widehat{f}(\xi, \eta) e^{2\pi i  (x \xi + y \eta)} d\eta d\xi\] 
is bounded in $L^{2}(\mathbb{R}^{2})$.
\end{theorem}

This theorem should be compared with the main result in \cite{GT2016} by Guo and the second author. It is an analogous estimate for the directional Hilbert transform and shows that 
if the Lipschitz continuity is enhanced by a lacunarity-type condition on $V$, the $L^{p}$ bounds do follow for all $1<p<\infty$. It is a well known open problem whether or not $V$ being Lipschitz alone is enough to imply $L^{p}$ boundedness of the truncated directional Hilbert transform. On the other hand, a counterexample by Karagulyan \cite{Karagulyan2007} shows that without the Lipschitz assumption, the mere assumption of $V$ taking values in a lacunary set such as $\{2^{k}\}_{k \in \mathbb{Z}}$ is not sufficient.  

%

It is tempting to investigate whether Karagulyan's counterexample can be modified to have some implications in the case of the hyperbolic cross multiplier, such as 
proving unboundedness for the maximal linearized hyperbolic cross multiplier
in the absence of a Lipschitz condition. However, we have not been able to implement this idea. Karagulyan's example uses the fact that every sector between two
lines emanating from the origin contains discs of arbitrarily large diameter. 
The set between two dilates of hyperbolas however do not contain arbitrarily 
large discs.

The rest of the present paper is structured as follows. After introducing some notation, we prove a dyadic model of our theorems. This part is parallel to
\cite{GT2016} and further highlights the connection to the directional Hilbert transform. Then we prove technical versions, Lemma \ref{lemma:main_lemma} and Lemma \ref{lemma:main_lemma2}, of the main theorems, and the last section is devoted to deriving the theorems from the technical Lemmas.

\smallskip

\textit{Acknowledgement.} The authors would like to thank Andreas Seeger for bringing this problem to their attention, and Joris Roos for suggesting to look at general $\beta$. Part of the research was done during the first author's stay at the Mathematical Institute of the University of Bonn, which he wishes to thank for its hospitality. The first author acknowledges support from the V\"ais\"al\"a Foundation, the Academy of Finland, and the NSF Grant no.\ DMS-1440140 (the paper was finished while in residence at MSRI, Berkeley, California, during the Spring 2017 semester). The second author acknowledges support from the Hausdorff center of Mathematics and DFG grant CRC 1060. 

\section{Notation}
\label{sec:notation}
Most of the notation we use is standard. For a Schwartz function $f$, we define the Fourier transform as 
\[\widehat{f}(\xi, \eta) = \iint_{\mathbb{R}^{2}} f(x,y) e^{-2\pi i (x \xi + y \eta)} dxdy. \]
We denote by $C$ the positive constants that only depend on parameters we do not keep track of, $p$ and $\beta$. Dependency on $m$ is also sometimes hidden. If $A \leq C B$ for positive numbers $A$ and $B$, we write $A  \lesssim B$. For a set $E$, the characteristic function $1_E$ takes the value $1$ in $E$ and equals zero elsewhere. For a measurable set $E$, we denote its Lebesgue measure, whose dimension is always clear from the context, by $|E|$.

\section{The dyadic models}
In this section, $I$ and $J$ will denote dyadic intervals, that is, intervals of the form $ 2^{k} ((0,1] + n )$ with $n,k \in \mathbb{Z}$ and $n\ge 0$. With each dyadic interval $I$ we associate the $L^{2}$ normalized Haar function 
\[h_I = |I|^{-1/2} (1_{I^{-}} - 1_{1^{+}})\]
where $I^\pm$ denote the left and right halves of $I$. We consider the basis $\{h_I \otimes h_J\}_{I,J}$ in the plane, and $\langle \cdot, \cdot \rangle$ means the $L^{2}(\mathbb{R}^{2})$ inner product. Occasional use of the inner product in one coordinate is denoted by a subscript.

The dyadic metric is defined by $d(x,x') = \inf \{|I|: x,y \in I \}$ where the infimum is over all dyadic intervals. The two dimensional dyadic metric is defined analogously with dyadic squares. The following theorem is the dyadic model for Theorem \ref{thm:main_theorem_2}.

\begin{theorem}\label{dyadictheorem}
Let $L > 0$, and let 
\[V:(0,\infty)^{2} \to \{ 2^{k} \}_{k \in \mathbb{Z}}, \quad V^{1/2} >  L\] 
be $L$-Lipschitz with respect to the dyadic metric. Then
\[f \mapsto  \sum_{ |I||J| \leq V(x,y) } \langle h_{I} \otimes h_J , f \rangle h_{I}(x) h_J(y)   \] 
with the sum over dyadic intervals $I$ and $J$ is a bounded operator from $L^{p}((0,\infty)^{2})$ into itself for all $1<p<\infty$. 
\end{theorem}
\begin{proof}


In analogy to \cite{GT2016}, the heart of the matter is the following proposition.

\begin{proposition}
\label{dyadicproposition}
Take $(x,y) \in (0,\infty)^{2}$. Let $I \ni x$ and $J \ni y$ with $|I| \leq |J|$. If $|I||J| \leq V(x,y)$, then $|I||J| \leq V(x',y)$ for all $x' \in I$.
\end{proposition}
\begin{proof}
If there existed $x' \in I$ with $|I||J|> V(x',y) $, then $V(x',y) < |I||J| \leq V(x,y)$ would hold. The assumptions together with this observation immediately imply
\begin{align*}
V(x,y) = d(V(x',y),V(x,y)) \leq L d(x,x') \leq L|I| \\
 \leq L \sqrt{|I||J|}
\leq L \sqrt{V(x,y)} < V(x,y),
\end{align*}
which is a contradiction. 
\end{proof}

We can divide the sum in the definition of the dyadic operator of Theorem 
\ref{dyadictheorem} into two parts and without loss of generality concentrate on
\begin{align*}
 \sum_{I}
\sum_{\substack{ J: |I||J| \leq V(x,y) \\
 |I| \leq |J| } }  \langle h_{I} \otimes h_J , f \rangle h_{I}(x) h_J(y).   
\end{align*}

Let $\mathcal{J}(I,y)$ be the collection of dyadic intervals $J$ with $y \in J $ and $|I| \leq |J|$ for which there exists $x$ so that $|I||J| \leq V(x,y)$. Suppose then that $J \in \mathcal{J}(I,y)$. If there exists $x \in I$ such that $|I||J| \leq V(x,y) $, the above proposition implies that this inequality is true
for all $x\in I$. Hence the inner sum can be viewed as sum over $J\in \mathcal{J}(I,y)$, in effect removing the $x$ dependency from the inner sum. We then estimate the $L^p$ norm of the last display by means of the dyadic Littlewood--Paley inequality:
\begin{multline*}
\iint \left \lvert 	\sum_{I } \sum_{ J \in  \mathcal{J}(I,y)  }  \langle h_{I} \otimes h_J , f \rangle h_{I}(x) h_J(y)		\right \rvert ^{p} dxdy   \\
\lesssim 
	\iint 	\left( 
	\sum_{I  }  	\left \lvert \sum_{ J \in  \mathcal{J}(I,y)  }  \langle h_{I} \otimes h_J , f \rangle h_{I}(x) h_J(y)		\right \rvert ^{2}
	\right)^{p/2}  dxdy . 
\end{multline*}
For a fixed $I$, if $J , J' \in \mathcal{J}(I,y)$ are nested, then also $J''$ with $J' \subset J'' \subset J$ is in $ \mathcal{J}(I,y)$. Hence the inner sum telescopes into the difference of two dyadic martingale averages in the second variable. We estimate the above display by
\begin{multline*}
\iint 	\left( 
	\sum_{I  }  	\left \lvert  M_{2} ( \langle h_{I} , f \rangle_1 h_{I}(x))	(y)	\right \rvert ^{2}
	\right)^{p/2}  dxdy \\
\lesssim
	\iint 	\left( 
	\sum_{I  }  	\left \lvert   \langle h_{I} , f(\cdot,y) \rangle_1 h_{I}(x)		\right \rvert ^{2}
	\right)^{p/2}  dy dx  
\lesssim
	\norm{f}_{L^{p}}^{p}
\end{multline*}
where $M_2$ is the dyadic maximal function with respect to the second variable; it is controlled by the vector valued Fefferman--Stein maximal function theorem. The last step in the above estimation relies again on the dyadic Littlewood--Paley inequality.
\end{proof}

To obtain a dyadic model of Theorem \ref{thm:main_theorem}, we replace the lower bound of $V$ by a truncation of the admissible scales. Note that the choice $\beta =  -1$ in the resulting theorem corresponds to the directional Hilbert transform (or its dyadic model to be precise).

\begin{theorem}\label{dyadictheorem_2}
Let $L> 0$ and $ \beta  \in \mathbb{R}$, and let 
\[V:(0,\infty)^{2} \to \{ 2^{k} \}_{k \in \mathbb{Z}} \] 
be $L$-Lipschitz in the first variable with respect to the dyadic metric. Then
\[f \mapsto  \sum_{ \substack{ |I||J|^{\beta} \leq V(x,y) \\ |J|^{\beta} \geq L}} \langle h_{I} \otimes h_J , f \rangle h_{I}(x) h_J(y)   \] 
is a bounded operator from $L^{p}((0,\infty)^{2})$ to itself for all $1<p<\infty$. 
\end{theorem}

The proof of the theorem is almost identical to that of Theorem \ref{dyadictheorem}. The only difference is that where we previously used Proposition \ref{dyadicproposition}, we will now use the following fact. 

\begin{proposition}
\label{dyadicproposition_2}
Take $(x,y) \in (0,\infty)^{2}$. Let $I \ni x$ and $J \ni y$. If $|I||J|^{\beta} \leq V(x,y)$, then $|I||J|^{\beta} \leq V(x',y)$ for all $x' \in I$.
\end{proposition}
\begin{proof}

If there were $x' \in I$ with $|I||J|^{\beta} > V(x',y)$, then $V(x',y) < |I||J|^{\beta} \leq V(x,y)$ would hold. Since $|J|^{\beta} \geq L$, we see that
\begin{align*}
V(x,y) = d(V(x',y),V(x,y)) \leq L d(x,x') \leq L|I| < |J|^{\beta} |I| \leq V(x,y),
\end{align*}
which is a contradiction. 
\end{proof}

Note that in this case the two variables behave differently. We did not assume any ordering of $|I|$ and $|J|$ so we need not decompose the operator to two sums with $|I| \leq |J|$ and $|J| < |I|$ in the proof of the theorem.

\section{The technical lemmas}

The following lemma captures the essence our main results.

\begin{lemma}
\label{lemma:main_lemma2}
Let $m \in C^{3}(\mathbb{R})$ satisfy $1_{(-\epsilon,\epsilon)} \leq m \leq 1_{(-2\epsilon,2\epsilon)}$ for dyadic $\epsilon \in (0,1]$. Let $L > 0$, $\beta \in \mathbb{R}$ and let $V : \mathbb{R}^{2} \to [0,\infty)$ satisfy
\[| V(x,y) - V(x,y')| \leq  L |y-y'|   \]
for all $x,y,y' \in \mathbb{R}$. We define
\[ T_{m,V} \Pi_{\beta} f (x,y) :=  \int_{\mathbb{R}} \int_{|\xi|^{\beta} \leq L^{-1}} m ( V(x,y) |\xi|^{\beta} |\eta| ) e^{2 \pi i (\xi x + \eta y)} d \xi d \eta . \]
Then 
\[\norm{T_{m,V} \Pi_{\beta} f }_{L^{p}(\mathbb{R}^{2})} \leq  C(p, \beta)A \left ( \log\frac{1}{\epsilon}  +1 \right )   \norm{f}_{L^{p}(\mathbb{R}^{2})}  \]
for all $1 < p < \infty$ and all Schwartz functions $f$. Here
\[A = \sum_{i=0}^{3} \sup_{t \in \mathbb{R}} |t^{i} m^{(i)}( t )| .  \]
\end{lemma}

\begin{proof}
Let $f$ be a Schwartz function. The proof consists of representing $f$ with the aid of the Calder\'on reproducing formulas; reducing the question to a cut-off of scales at the price of a Marcinkiewicz multiplier, and finally proving the theorem for the scale cut-off. For simplicity, we write $T$ for $T_{m,V}$ throughout the proof. We start the proof by first assuming that $\epsilon = 1$. 

\subsection{Calder\'on formulae}
Let $\phi_1, \phi_2 \in \mathscr{S}(\mathbb{R})$ be even and real. Let $\widehat{\phi_1}$ be supported in the annulus $[-2^{1/ |\beta|},2^{1/|\beta|}] \setminus (-2^{-1/ |\beta|},2^{-1/|\beta|})$ and $\widehat{\phi_2}$ in $[-2,2] \setminus (-1,1)$. Let $\psi_2$ with \textit{space support} in $(-2^{-5}, 2^{-5})$ be one more Schwartz function. We require the normalization 
\[1 = \int_{0}^{\infty}\widehat{\phi_1}(t) \frac{dt}{t} = \int_{0}^{\infty}\widehat{\phi_2}(t)\widehat{\psi_2}( t ) \frac{dt}{t} \]
and the mean zero condition $\widehat{\phi_2} (0) = \widehat{\psi_2}(0) = 0$.

A subscript as in $\phi_{1t}(x) = t\phi_{1}(tx)$ for $t > 0$ denotes a dilation, and the same notation is also used for $\phi_2$ and $\psi_2$. We define the Littlewood-Paley operators $P_{1s}$, $P_{2t}$ and $P_{3t}$ for $s,t > 0$ by
\begin{align*}
P_{1s}f(x,y) &=  \int_{\mathbb{R}} \phi_{1 s}(x-z)f(z,y) dz, \\
P_{2t}f(x,y) &=  \int_{\mathbb{R}} \phi_{2 t}(y-z)f(x,z) dz, \\
P_{3t}f(x,y) &=  \int_{\mathbb{R}} \psi_{2 t}(y-z)f(x,z) dz.
\end{align*}
Under the conditions stated above, the Calder\'on reproducing formulae corresponding to each variable hold true:
\begin{align*}
f(x,y) = \int_{0}^{\infty} P_{1s}f(x,y) \frac{ds}{s} = \int_{0}^{\infty} P_{2t} P_{3t}f(x,y)  \frac{dt}{t}.
\end{align*}
In particular, we can write
\[f(x,y) = \int_{0}^{\infty} \int_0^{\infty} P_{1s}P_{2t}P_{3t} f(x,y) \frac{dt}{t} \frac{ds}{s} .\]

\subsection{Decomposition of the operator}
Consider $\lambda > 0$. We let $\widehat{T_\lambda}(\xi,\eta) = m(\lambda |\xi \eta|)\widehat{f}(\xi,\eta)$. Taking the Fourier transform and using the Calder\'on formula, we note that
\begin{align*}
&\widehat{T_\lambda f}(\xi, \eta) = m (\lambda |\xi \eta|  ) \widehat{f}( \xi ,\eta ) \\
	&= \int_{0}^{\infty} \int_0^{\infty} m (\lambda |\xi \eta|  ) \widehat{\phi_{1}} ( \xi/s)\widehat{\phi_{2}}(\eta/t  )\widehat{\psi_{2}} (\eta/t )  \widehat{f}(\xi, \eta) \frac{dt}{t} \frac{ds}{s} \\
	&= \int_{0}^{\infty} \int_0^{4   /( \lambda s )} m (\lambda |\xi \eta|  ) \widehat{\phi_{1}} ( \xi/s)\widehat{\phi_{2}}( \eta/t  )\widehat{\psi_{2}}(\eta/t ) \widehat{f}(\xi, \eta) \frac{dt}{t} \frac{ds}{s}      .
\end{align*}	
The discarded part of the $t$-integral gives zero contribution because the following conditions	
\begin{itemize}
\item $\widehat{\phi_{2}}(\eta/t) \neq 0 $ only if $\frac{1}{2}   |\eta| \leq t \leq   |\eta|$,
\item $\widehat{\phi_{1}}(\xi/s) \neq 0 $ only if $2^{- 1 / |\beta| }  |\xi| \leq s \leq 2^{1/|\beta|} |\xi|$,
\item $m ( \lambda |\xi \eta|  ) \neq 0 $ only if $|\xi |^{\beta} | \eta| < 2   \lambda^{-1}$,
\end{itemize}
imply that whenever the integrand of the last display does not vanish we have
\[t \leq |\eta| \leq \frac{2   }{\lambda |\xi|^{\beta} } \leq \frac{ 4 }{\lambda s^{\beta} } .\]
Using the fact that $m (\lambda |\xi|^{\beta} | \eta| ) = 1$ when $|\xi |^{\beta} | \eta| <   \lambda^{-1}$ (in particular when $ 4s^{\beta} t \leq \lambda^{-1}$), we can write
\begin{align*}
\widehat{T_\lambda f}(\xi, \eta)   &=  \widehat{S_\lambda f}(\xi, \eta) + \widehat{E_\lambda f}(\xi, \eta).
\end{align*}
These operators are defined as follows. We let $\tilde{\lambda} = \min \{2^{j + 2} : 2^{j} >  \lambda, j \in \mathbb{Z}\}$. Then
\begin{align*}
S_\lambda f(x,y) &= \int_{0}^{\infty} \int_0^{(\tilde{\lambda} s^{\beta} )^{-1}} P_{1s}P_{2t} P_{3t} f(x,y) \frac{dt}{t} \frac{ds}{s} ,
\end{align*}
is the principal term and 
\begin{equation*}
\widehat{E_\lambda f} = \int_{0}^{\infty} \int_{  (\tilde{ \lambda } s^{\beta})^{-1}}  ^{4  (\lambda s^{\beta} )^{-1}} m (\lambda |  \xi \eta|  ) \widehat{\phi_{1}} (\xi/s)\widehat{\phi_{2}}(\eta / t )\widehat{\psi_{2}}(\eta/t) \widehat{f}(\xi, \eta) \frac{dt}{t} \frac{ds}{s}    
\end{equation*}
is an error term.

A similar decomposition is valid with $\lambda$ replaced by $V(x,y)$. For every integer $j$, we denote 
\[\Omega_j = \{ (x,y) \in \mathbb{R}^{2}: 2^{j }  \leq V(x,y) < 2^{j+1}  \} .\] 
We have that 
\[Tf = Sf  + Ef \]
where
\begin{equation}
1_{\Omega_j} S f (x,y) = 1_{\Omega_j}(x,y) \cdot \int_{0}^{\infty} \int_0^{ (2^{j+3} s^{\beta})^{-1} } P_{1s}P_{2t} P_{3t} f(x,y) \frac{dt}{t} \frac{ds}{s}
\end{equation}
is the principal part of the operator, and the error part $Ef$ is defined as  
\begin{equation}
\label{eq:error_term}
1_{\Omega_j} E f(x,y)  =  \int_{0}^{\infty} \int_{ (2^{j+3} s^{\beta})^{-1}} ^{ 4   ( V(x,y) s^{\beta})^{-1}}  T P_{1s} P_{2t} P_{3t} f  (x,y) \frac{dt}{t} \frac{ds}{s}
\end{equation}
for all $(x,y) \in \Omega_j$. Further, it is possible to write
\begin{align*}
Ef(x,y) &=   E_{2^{j}}f(x,y) + \int_{2^{j}}^{V(x,y)} \partial_\tau E_\tau f(x,y)d \tau  
\\
 		&\leq  |E_{2^{j}}f(x,y)| + \int_{2^{j}}^{2^{j+1}}   | \partial_\tau E_\tau f(x,y) | d\tau 
.
\end{align*}
We call these two components large variation error and small variation error
and estimate them separately.

\subsection{The large variation error}
\label{subse:lv}
The multiplier $E_{2^{j}}$ has symbol
\[ \int_{0}^{\infty} \int_{  ( 2^{j+3} s^{\beta} )^{-1}} ^{  ( 2^{j-2 } s^{\beta} )^{-1}}  m (2^{j} |  \xi \eta|  ) \widehat{\phi_{1}} (\xi /s )\widehat{\phi_{2}}(\eta /t)\widehat{\psi_{2}}(\eta/t) \frac{dt}{t} \frac{ds}{s}    \]
that is supported in 
\[\{(\xi,\eta) \in \mathbb{R}^{2} :   2^{- j - 5} \leq |\xi |^{\beta} |\eta| \leq   2^{-  j + 4} \}.\]
Since $\Omega_j$ are pairwise disjoint, we see that
\[\norm{\sum_{j \in \mathbb{Z}} 1_{\Omega_j} E_{2^j} f }_{L^{p}(\mathbb{R}^{2})} \leq \norm{ (\sum_{j \in \mathbb{Z}} | E_{2^{j}}f |^{2} )^{1/2}} _{L^{p}(\mathbb{R}^{2})} .\]
As $j$ runs over all integer values, the number of $E_{2^{j}}$ having non-zero symbol at any point in the frequency plane is uniformly bounded. Hence one can verify that $\sum_j \alpha_j E_{ 2^j}$ is a Marcinkiewicz multiplier with uniform symbol estimates for any choice of $|\alpha_j| \leq 1$. 

By the Marcinkiewicz multiplier theorem and Khinchin's inequality
\[\norm{ (\sum_{j \in \mathbb{Z}} | E_{2^{j}} f |^{2} )^{1/2}} _{L^{p}(\mathbb{R}^{2})} \lesssim  \sup_{(\alpha_j)_{j \in \mathbb{Z}}} \norm{ \sum_{j \in \mathbb{Z}} \alpha_j E_{2^{j}} f  }_{L^{p}(\mathbb{R}^{2})} \leq A   \norm{f}_{L^{p}(\mathbb{R}^{2})} \]
which proves the claimed bound for the large variation error. 

\subsection{The small variation error}
\label{subse:sv}
To estimate
\[\left \lVert \sum_{j} 1_{\Omega_j} \int_{2^{j}}^{2^{j+1}}   | \partial_\tau E_\tau f(x,y) | d\tau \right \rVert_{L^{p}(\mathbb{R}^{2})} , \]
we write $E'_\tau$ for the operator $\partial_\tau E_\tau f(x,y)$. Changing variables in the $\tau$-integral, using disjointness of $\Omega_j$, and applying Minkowski's inequality, we obtain
\begin{multline*}
\left \lVert \int_{1}^{2}  | \sum_{j}  1_{\Omega_j}  2^{j}  E_{2^{j} \tau}' f(x,y) | d\tau \right \rVert_{L^{p}(\mathbb{R}^{2})} \\
 \leq  \int_{1}^{2}  \left \lVert (\sum_{j}   |2^{j}\tau E_{2^{j} \tau}' f(x,y) |^{2} )^{1/2} \right \rVert_{L^{p}(\mathbb{R}^{2})} \frac{d\tau }{\tau}  .
\end{multline*}
The symbol of the multiplier $2^{j}\tau E_{2^{j} \tau}'$ is supported in
\[\{(\xi,\eta) \in \mathbb{R}^{2} :  2^{- j - 5 } \tau   \leq |\xi |^{\beta} | \eta |\leq  2^{-  j + 4} \tau   \},\]
and it satisfies the estimates required for the Marcinkiewicz multiplier theorem uniformly in $2^{j}\tau$. Hence
\[ \left \lVert \sum_{j}   \alpha_j 2^{j}\tau E_{2^{j} \tau}' f(x,y)  \right \rVert_{L^{p}(\mathbb{R}^{2})} \leq A  \]
with $A$ independent of $| \alpha_j | \leq 1$ so we can again use Khinchin's inequality to conclude the desired bound. The use of Marcinkiewicz multiplier theorem here and in the previous subsection yields the constant $A$ appearing in the claim. See e.g.\ \cite{Grafakos2008}.

\begin{remark}
\label{remark:lacunary}
If we were only to prove $L^{2}$ bounds and we assumed that $V$ took values in a lacunary set like $\{2^{k}\}_{k \in \mathbb{Z}}$, we could work with a non-smooth $m$ like $1_{[0,1)}$. To handle this case, we note that $\Omega_j = \{ V(x,y) = 2^{j} \}$ and that, consequently, there is no small variation error term. The operators $E_{2^{j}}$ have bounded symbols and bounded overlap as $j$ takes all integer values. Hence we can use orthogonality to conclude  
\[\norm{Ef}_{L^{2}} = \norm{ \sum_j 1_{\Omega_j} E_{2^{j}} f }_{L^{2}}  \leq \norm{ \sum_j E_{2^{j}} f }_{L^{2}} \lesssim \norm{f}_{L^{2}}. \]
This observation is needed to prove Theorem \ref{thm:3}. The rest of its proof is identical to the proof of Lemma \ref{lemma:main_lemma2} (the current Lemma with Remark \ref{remark:lipschitz}).
\end{remark}

\subsection{The principal term}
Let $v(x,y) = \min \{2^{j+2}: 2^{j} >  V(x,y), j \in \mathbb{Z}\} $. We will estimate
\begin{align*}
Sf(x,y) 
	&= \int_{0}^{\infty} \int_0^{   (v(x,y)s^{\beta})^{-1} } P_{1s}P_{2t} P_{3t} f(x,y) \frac{dt}{t} \frac{ds}{s} .
\end{align*}
Take a Schwartz function $g$. Aiming at a bound only depending on $\norm{g}_{L^{p'}}$ and $\norm{f}_{L^{p}}$, we write the $P_{3t}$ convolution out and divide the $t$-integral into two parts
\[ \iint_{\mathbb{R}^{2}}		 Sf(x,y) g(x,y) dxdy = I + II \]
where
\begin{align*}
I &= \int_{\mathbb{R}^{3}}   \int_{0}^{\infty} \int_{0}^{  (v(x,z)s^{\beta})^{-1}}  P_{1s}P_{2t }f(x,z) \psi_{2t }(y-z)g(x,y) \frac{dt}{t} \frac{ds}{s}  dzdxdy, \\
II &= \int_{\mathbb{R}^{3}}   \int_{0}^{\infty} \int_{  (v(x,z)s^{\beta})^{-1} }^{  (v(x,y)s^{\beta})^{-1}}  P_{1s}P_{2t }f(x,z) \psi_{2t }(y-z)g(x,y) \frac{dt}{t} \frac{ds}{s}  dzdxdy .
\end{align*}
We start by estimating the term $II$. Change the $y$ integration to take place before $x$- and $z$-integrations. The domain of $y$-integration is divided into two parts, in one of which $ (v(x,z)s)^{-1}  <  (v(x,y)s)^{-1}$. We will start with this part.

\subsection{The role of the Lipschitz assumption }
\label{sec:lipschitz}
Since $v$ is defined as the dyadic value with $2^{-3} v \leq V < 2^{-2} v $, we see that also  $V(x,z)^{-1} \leq V(x,y)^{-1}$ holds. We have assumed that $|\xi|^{\beta} \leq L^{-1}$ so also $L \leq |\xi|^{-\beta} \leq 2 s^{-\beta}$. By the Lipschitz assumption and the fact that $\supp \psi_2 \subset (-2^{-4}/3,2^{-4}/3)$, we see that
\[| V(x,y) - V(x,z) | 	\leq L|y-z| \leq \frac{2^{-4} L }{ 3t } \leq \frac{ 2^{-3} }{3 s^{\beta} t}  \leq  \frac{1}{24} v (x,z) \leq \frac{1}{3} V(x,z).  \]
In particular
\[  \frac{V(x,z)}{V(x,y)} \leq 1 + \frac{ |V(x,z) - V(x,y) |}{V(x,y)} = 1 + \frac{1}{3} \frac{V(x,z)}{V(x,y)}  \]
so
\begin{equation}
\label{eq:lipschitz_ratio}
1 \leq \frac{V(x,z)}{V(x,y)} \leq \frac{3}{2}
\end{equation}
An identical argument also yields the same result in the set where $v(x,y) > v(x,z)$.

\begin{remark}
\label{remark:lipschitz}
In case we assume that $V$ is bounded from below by $L^{2}$; the Lipschitz condition only holds for $|y-z| \geq L$; and $\beta = 1$, we argue as follows. We look again at the case $V(x,z)^{-1} \leq V(x,y)^{-1}$. By symmetry, we may assume that $\widehat{f}$ is supported in the double cone:
\begin{equation}
\label{assumption:fourier-support}
\supp \widehat{f} \subset \{(\xi,\eta) \in \mathbb{R}^{2}: | \xi| \leq  | \eta | \}.
\end{equation}
Since $\supp \psi_2 \subset (-2^{-4}/3,2^{-4}/3)$ , we see that
\begin{align*}
| V(x,y) - V(x,z) | 	&\leq L|y-z| \leq \frac{L}{ 48 t }
\end{align*}
whenever $L|y-z| \geq L^{2}$. As a consequence of restriction to the cone in \eqref{assumption:fourier-support}, the values of $s$ and $t$ that have a non-zero contribution in the integral above satisfy $s \leq 2 |\xi| \leq 2 |\eta| \leq 4 t $. Now
\[\frac{1}{t} \leq 2  \frac{1}{\sqrt{ts} }  \leq 2 \sqrt{v(x,z)} \leq 8 \sqrt{V(x,z)} \]
and by the lower bound $L \leq \sqrt{V(x,z)}$, we see that
\[\frac{L}{ 48t } \leq \frac{1}{6}V(x,z).\]
Then
\[  \frac{V(x,z)}{V(x,y)} \leq 1 + \frac{ |V(x,z) - V(x,y) |}{V(x,y)} = 1 + \frac{1}{3} \frac{V(x,z)}{V(x,y)}  \]
so we have \eqref{eq:lipschitz_ratio}. It is clear also in case $L|y-z| < L^{2} \leq V(x,y)$. The same argument works in the set where $v(x,y) > v(x,z)$, so we conclude that inequality  \eqref{eq:lipschitz_ratio} holds in this setting as well.
\end{remark}

\subsection{Estimate for $II$}
Recall that $v = \min \{2^{j+2}: 2^{j} > V, j \in \mathbb{Z}\}$. By inequality \eqref{eq:lipschitz_ratio}, we can divide the integration over $x$, $y$ and $z$ in the definition of $II$ into two sets:
\begin{align*}
E_{\pm} = \{(x,y,z): 2^{\pm 1} v(x,z) = v(x,y) \}.
\end{align*}
The remaining alternative $v(x,y) = v(x,z)$ gives a zero contribution. Clearly $(x,y,z) \in E_+$ if and only if there exists a unique integer $k_{xyz}$ such that
\[  V(x,z) < 2^{k_{xyz}} \leq  V(x,y). \]
Under the condition \eqref{eq:lipschitz_ratio}, the existence of such an integer is equivalent to the following two conditions  
\begin{align*}
(x,z) &\in F_1 := \{ (x,z): (V(x,z), 3 V(x,z)/2 ] \cap D \neq \varnothing \} \quad \textrm{and} \\
(x,y) &\in F_2 := \{ (x,y): [2V(x,y)/3,V(x,y)] \cap D \neq \varnothing\}
\end{align*}
with $D = \{2^{k}\}_{k \in \mathbb{Z}}$ holding simultaneously.

Hence
\begin{multline*}
\bigg \lvert \int_{\mathbb{R}^{3}} 1_{E_{+}}\int_{0}^{\infty}   \int_{  (v(x,z) s^{\beta})^{-1} }^{  (v(x,y)s^{\beta})^{-1}}  P_{1s }P_{2t}f(x,z) \psi_{2t}(y-z)g(x,y) \frac{dt}{t}\frac{ds}{s} dydxdz \bigg \rvert  \\
=\bigg \lvert \int_{\mathbb{R}^{3}} 1_{F_1}(x,z) 1_{F_2}(x,y) \int_{0}^{\infty}   \int_{  (v(x,z) s^{\beta})^{-1} }^{  (2v(x,z)s^{\beta})^{-1}} P_{1s }P_{2t}f(x,z) 
	  \\ \cdot   \psi_{2t}(y-z)g(x,y) \frac{dt}{t}\frac{ds}{s} dydxdz \bigg \rvert   \\
= \left \lvert  \int_{\mathbb{R}^{2}} 1_{F_1} \int_{0}^{\infty}  \int_{  (v(x,z)t)^{-1/\beta} }^{  (2v(x,z)t)^{-1/\beta}}   P_{1s}P_{2t}f(x,z)  \frac{ds}{s}  P_{3t}( g 1_{F_2} )(x,z)    \frac{dt}{t} dxdz \right \rvert \\
\lesssim 	\int_{\mathbb{R}^{2}}  \int_{0}^{\infty} M_1 (P_{2t}f )(x,z) |P_{3t}(g 1_{F_2} )(x,z)|    \frac{dt}{t} dzdx \\
\leq 	\left \lVert \left( \int_{0}^{\infty} |M_1 (P_{2t}f )|^{2} \frac{dt}{t} \right)^{1/2} \right \rVert    _{L^{p}(\mathbb{R}^{2})} \left \lVert    \left( \int_{0}^{\infty}   | P_{3t}( g 1_{F_2} )|^{2}    \frac{dt}{t} \right)^{1/2}      \right \rVert    _{L^{p'}(\mathbb{R}^{2})} .
\end{multline*}

The first factor is bounded by a square function estimate and an application of the $ L^{2}(\mathbb{R}_{+}, dt/t)$ valued Fefferman--Stein maximal function theorem (see the book \cite{Stein1993}). We denoted by $M_1$ the Hardy--Littlewood maximal function on the first variable. The second factor is bounded by the square function estimate. The same argument applies if we replace $E_+$ by $E_-$. Hence we have bounded the term $II$.

\subsection{Estimate for $I$}
The integration limits in the definition of $I$ do not depend on the variable $y$. Hence we can skip the argumentation involving the Lipschitz assumption and directly apply the previous computation. For $\beta > 0$, This results in the bound
\begin{multline*}
 I \leq \iint_{\mathbb{R}^{2}}   \int_{0}^{\infty} \int_{0}^{  (v(x,z)s^{\beta})^{-1}}  P_{1s }P_{2t}f(x,z) P_{3t}g(x,z) \frac{dt}{t} \frac{ds}{s}  dzdx \\
 	 \leq \left \lVert \left(       \int_{0}^{\infty}   \left \lvert \int_{0}^{   (v(x,z)t)^{-1/\beta}} P_{1s} P_{2t}f  \frac{ds}{s} \right \rvert ^{2}  \frac{dt}{t} \right)^{1/2} \right \rVert    _{L^{p}(\mathbb{R}^{2})} \\
  \times	 \left \lVert    \left( \int_{0}^{\infty}   | P_{3t} g |^{2}    \frac{dt}{t} \right)^{1/2}      \right \rVert    _{L^{p'}(\mathbb{R}^{2})} \\
  \leq \left \lVert \left(       \int_{0}^{\infty}  |  M_1 P_{2t}f | ^{2}  \frac{dt}{t} \right)^{1/2} \right \rVert    _{L^{p}(\mathbb{R}^{2})} \left \lVert    \left( \int_{0}^{\infty}   | P_{3t} g |^{2}    \frac{dt}{t} \right)^{1/2}      \right \rVert    _{L^{p'}(\mathbb{R}^{2})}.
\end{multline*}
The last inequality follows from the fact that since
\[x \mapsto \int_{0}^{1} \phi_{1s}(x) \frac{ds}{s} \]
is a Schwartz function, the corresponding convolution is controlled by the Hardy--Littlewood maximal function, and the rest of the estimate is justified by a simple change of variable. Now the estimation can be concluded as in the previous case by Fefferman--Stein and Littlewood--Paley inequalities. The case $\beta < 0$ is similar except for the second line that has integration in the interval $(v(x,z) t)^{-1/\beta} ,\infty)$. However, the same bound holds true.  

\subsection{General support}
We have proved the claim for $\epsilon = 1$. To relax this assumption, take $\epsilon = 2^{-i}$ with $i \in \mathbb{Z}_{+}$ and take $m$ as in the statement of the Lemma. Take positive and smooth functions $\varphi_j$ with $\supp \varphi_j \subset (2^{-j-1},2^{-j+1})$ for every integer $0 \leq j < i$ and let
\[\tilde{m} = \sum_{j=0}^{i-1} \varphi_j + m \]
so that $1_{[0,1)} \leq \tilde{m} \leq 1_{(-2,2)}$ can be continued to be an even function that satisfies the assumptions for which we have already proved the result. Now
\[T_{m,V} = T_{\tilde{m},V} - \sum_{j=0}^{i-1} T_{\varphi_j,V} .\]
The operator norm of the first term we have already bounded. The remaining terms are of the form \eqref{eq:error_term}, that is,
\[T_{\varphi_j,V}f(x,y) =  \int_{0}^{\infty} \int_{ (4  V(x,y) s^{\beta})^{-1}} ^{4   ( V(x,y) s^{\beta})^{-1}}  T_{\varphi_j,V}f P_{1(2^{-j}s)} P_{2t} P_{3t} f  (x,y) \frac{dt}{t} \frac{ds}{s}\]
and, accordingly, they admit a bound comparable to the absolute constant
\[\sum_{k=0}^{3} \sup_{t\in (2^{-j-1},2^{-j+1})} |t^{k}\varphi_j^{(k)}(t)| .\] 
This was proved in subsections \ref{subse:lv} and \ref{subse:sv}. In total,
\[\norm{T_{m,V}} \lesssim i = - \log \epsilon.\]
\end{proof}

Next we formulate another version of the Lemma that will be needed to prove Theorem \ref{thm:main_theorem_2}.

\begin{lemma}
\label{lemma:main_lemma}
Let $m \in C^{3}(\mathbb{R})$ satisfy $1_{(-\epsilon,\epsilon)} \leq m \leq 1_{(-2\epsilon,2\epsilon)}$ for dyadic $\epsilon \in (0,1]$. Let $L,V_{min} > 0$ and let $V : \mathbb{R}^{2} \to [V_{min},\infty)$ satisfy
\[| V(z) - V(z')| \leq \max( V_{min} , L |z-z'|  ) \]
for all $z,z' \in \mathbb{R}^{2}$. We define
\[ T_{m,V} f (x,y) :=  \iint_{\mathbb{R}^{2}} m ( V(x,y) |\xi| | \eta| ) \widehat{f}(\xi, \eta) e^{2 \pi i (\xi x + \eta y)} d \xi d \eta  .\] 
Then 
\[\norm{T_{m,V}f }_{L^{p}(\mathbb{R}^{2})} \leq  C(p)A \left \lceil \log\frac{1}{\epsilon} \right \rceil \left \lceil \log \frac{ L }{ V_{min}^{1/2}}  \right \rceil \norm{f}_{L^{p}(\mathbb{R}^{2})}  \]
for all Schwartz functions $f$ and $p > 1$. Here $\lceil \cdot \rceil = \inf \{ k \in \mathbb{Z}: k > 0, k > \cdot \}$ and
\[A = \sum_{i=0}^{3} \sup_{t \in \mathbb{R}} |t^{i} m^{(i)}( t )| . \]
\end{lemma}

\begin{proof}
Note that the proof of the previous Lemma together with the included Remark \ref{remark:lipschitz} proves the Lemma for $V_{min} \geq L^{2}$. To get rid of this restriction, we note that in the complementary case it is possible to write
\[T_{m,V} = \sum_{j=1}^{k} 1_{ \{2^{j-1}V_{min} \leq V < 2^{j} V_{min}\} } T_{m,V}  + 1_{ \{2^{k}V_{min} \leq V \}} T_{m,V}   \]
with $k$ being the integer to satisfy $2^{k-1} V_{min} \leq  L^{2} < 2^{k} V_{min}$. For the operators 
\[ 1_{ \{2^{j-1}V_{min} \leq V < 2^{j} V_{min}\} } T_{m,V}   \]
the inequality \eqref{eq:lipschitz_ratio} is trivially satisfied so we conclude that they are bounded with a uniform constant. Hence
\[\norm{T_{m,V}} \lesssim - k \log \epsilon \lesssim \log(L^{2} / V_{min} ) \log (1/\epsilon).\]
This completes the proof of this variant.
\end{proof}

\section{The main theorems}
Next we discuss how to get the Theorems formulated in the introduction from the technical Lemmas.

\begin{proof}[Proof of Theorems \ref{thm:main_theorem} and \ref{thm:main_theorem_2}]
To prove these theorems, it remains to explain how to replace the function $m$ of the previous Lemmas by a more general function. We let $m$ first be a $C^{3}(\mathbb{R})$ function that is positive and decreasing in $[0,\infty)$. In addition, we require that $m$ has support contained in $[-2,2]$ and $m(0) = 1$. We divide $m$ into layers: Let
\begin{align*}
m_1 &= \min(m , m(2^{-1})  ) \\
m_i &= \min( m - m_{i-1}, (m - m_{i-1})(2^{-i})), i \in \{2,3,\ldots \}.
\end{align*}
Summing the layers $m_i$ up, we recover $m$. Moreover, adding and subtracting a small correction term at each point $2^{-i}$, the same decomposition continues holding true with smooth $m_i$. The constant $A$ from Lemma \ref{lemma:main_lemma} corresponding to each $m_i / \norm{m_i}_{\infty}$ is uniformly bounded because of smoothness and compact support. Applying Lemma \ref{lemma:main_lemma} (or Lemma \ref{lemma:main_lemma2}) with $\epsilon = 2^{-i}$, we see that 
\[\norm{T_{V,m}}_{L^{p} \to L^{p}} \leq \sum_{i = 1}^{\infty} \norm{T_{V,m_i}}_{L^{p} \to L^{p}}   \lesssim \sum_{i=1}^{\infty} i \norm{m_i}_{\infty} \lesssim  \sum_{i=1}^{\infty} i 2^{-i} < \infty.\]
By linearity, the same result holds for all compactly supported functions in $C^{3}(\mathbb{R})$.
\end{proof}

\begin{proof}[Proof of Theorem \ref{thm:3}]
Under the assumptions of this theorem, we can mostly follow the proof of Lemma \ref{lemma:main_lemma2}. We decompose the operator as usually. Recalling Remark \ref{remark:lacunary}, we see that the error part of the operator is easily controlled. On the other hand, the principal term is essentially the same as in the proof of Lemma \ref{lemma:main_lemma} so we are done. 
\end{proof}

%

\end{document}